%% file: sADMM_main.tex
\documentclass[11pt]{article}
\usepackage[a4paper, total={6in, 9in}]{geometry}
\usepackage{amsthm}

\usepackage{authblk}
\usepackage{cite}
\usepackage{graphics} 
\usepackage{epsfig} 
\usepackage{times} 
\usepackage{amsmath} 
\usepackage{amssymb}  
\usepackage{subcaption}
\usepackage{float}
\usepackage{algorithm}
\usepackage{algpseudocode}
\usepackage{xspace}
\usepackage{verbatim}
\usepackage{multirow}
\usepackage{csquotes}
\usepackage{booktabs}
\usepackage{enumerate}  
\usepackage{cite}

\newcommand{\algrulehor}[1][.2pt]{\par\vskip.5\baselineskip\hrule height #1\par\vskip.5\baselineskip}
\newtheorem{theorem}{Theorem}

\newtheorem{lemma}{Lemma}
\newtheorem{remark}{Remark}
\newtheorem{defn}{Definition}
\newtheorem{assume}{Assumption}

\begin{document}
	\title{Sensitivity Assisted Alternating Directions Method of Multipliers for Distributed Statistical Learning}
	\author{Dinesh Krishnamoorthy$^{1}$, 
		 Vyacheslav Kungurtsev$^2$
		\thanks{Financial support from the Research Council of Norway through the IKTPLUSS programme is gratefully acknowledged by the first author and the OP VVV project
			CZ.02.1.01/0.0/0.0/16\_019/0000765 Research Center for Informatics by the second. }
		\thanks{$^1$Dinesh Krishnamoorthy is with the Harvard John A. Paulson School of Engineering and Applied Sciences, Cambridge, MA, 02138
			{\tt\small dkrishnamoorthy@seas.harvard.edu}}
		\thanks{$^2$Vyacheslav Kungurtsev is with the Department of Computer Science, Czech Technical University
			Prague, Czech Republic
			{\tt\small vyacheslav.kungurtsev@fel.cvut.cz}}
}	
	
	\maketitle

\begin{abstract}
 This paper  considers the problem of distributed learning using the alternating directions method of multipliers (ADMM).
 ADMM splits the learning
 problem into several smaller subproblems, usually by partitioning
 the data samples. The different
 subproblems can be solved in parallel by a set of worker computing
 nodes coordinated by a master node, and the subproblems are repeatedly
 solved until some stopping criteria is met. At each iteration, the worker nodes
 must solve an optimization problem whose difficulty
 increases with the size of the problem. In this paper, we propose
 a sensitivity-assisted ADMM algorithm that leverages
 the parametric sensitivities such that the subproblem's solutions
 can be approximated using a tangential predictor, thus easing the
 computational burden to computing one linear solve per iteration. The convergence properties of
 the proposed sensitivity-assisted ADMM algorithm is studied, and the numerical performance
 of the algorithm
 is illustrated on a nonlinear regression as well as a classification problem. 
 
\end{abstract}

\section{Introduction}\label{sec:intro}
Machine learning algorithms that aim to build predictive models using large data sets are prevalent across many fields ranging from medical diagnosis, telecommunications, financial services, image and speech recognition, social media, transport, robotics, energy and smart manufacturing to name a few. Particularly, in the context of learning and control, it may be desirable to learn either the  model, value function or the policy using large data sets. 
In many applications,  such as Internet-of-things, collaborative robots, social networks, wireless sensor networks, or cloud computing applications, large amounts of data are collected and stored in a distributed manner, and we often want to build a global inferential model that leverages all available data without the need for data sharing.
In other cases, the data set may be  so large that it cannot be processed by a single machine.  In such cases, distributed computing is a natural solution, where the learning problem is decomposed into smaller subproblems that can be solved in multiple processors using local training data set (which is only locally stored). The subproblems are then coordinated by a central server to find the solution to the large-scale problem to develop a global inferential model. This is commonly known as collaborative machine learning, or federated learning \cite{mcmahan2017communication}. 


There are two popular classes of distributed learning approaches, namely, \textit{gradient-based }approach and \textit{decomposition-coordination} approach \cite{tutunov2019distributed}. 
In the gradient-based approach, each agent takes gradient descent steps in parallel using its local data, and the central server takes a weighted average of the local models at each iteration. Distributed learning using stochastic gradient descent, such as \texttt{FedAvg} and \texttt{FedSGD} presented in \cite{mcmahan2017communication} belongs to this class of approach, where the central server combines the local stochastic gradient descent (SGD) to perform model averaging. Distributed approximate Newton-type method (DANE) presented in \cite{shamir2014communication} is also based on  model-averaging, where instead of the gradient descent step, the subproblems implicitly exploit their local Hessian information. However, many of the gradient-based approaches are mostly limited to smooth and strongly convex problems. 
This may be restrictive in many learning problems, which may have nonconvex terms (e.g. when fitting neural networks) and/or non-differentiable terms (e.g. when using $ \ell_1 $ regularization). 

The \textit{decomposition-coordination}  approach relies on dual methods, where the subproblems solve constrained optimization problems by relaxing the coupling constraints, and a central  coordinator updates the shared dual variables.
Alternating Directions Method of Multipliers (ADMM)  is a state-of-the-art from this class, where the large-scale optimization problem is decomposed into smaller subproblems, each solved by a worker computing node and then coordinated by a master computing node. The subproblems and the master's update are iteratively solved to find the solution to the original problem \cite{boyd2011ADMM}. ADMM has been shown to be a powerful and robust approach for distributed optimization well suited for several model fitting problems, including regularized regression such as ridge regression, Lasso regression, logistic regression as well as classification problems such as  support vector machines  \cite{boyd2011ADMM,forero2010consensus,mateos2010distributed,cao2020differentially}.  ADMM is also well suited when the learning problem is non-differentiable, e.g. when using $ \ell_1 $ regularization.
To this end, ADMM involves repeatedly solving a set of subproblems in a distributed manner and a computation to update the shared dual variables on the master node.

Repeatedly solving the subproblems from one iteration to the next can quickly add to the overall computation time. This is especially the case for large-scale problems, or ADMM formulations with several subproblems, which can require several iterations to converge to the overall optimal solution. Although there have been several developments in the distributed optimization literature, these have predominantly been focused on the master problem formulation, such that the number of iterations can be required can be reduced. As such the computation burden of solving the subproblems itself still remains. 

This is further accentuated by the fact that in most cases the coordination between the master and the subproblems are  performed under a synchronous protocol, which leads to memory \enquote{locking}. That is,  the master updates are performed only after all the subproblems are solved at each iteration. Such memory locking implies that  the computation time of each iteration is limited by the \enquote{slowest} subproblem. This is schematically illustrated in Fig.~\ref{Fig:MemoryLocking}.   This is especially the case in heterogeneous problems where the different subproblems can  vary in problem size and complexity resulting in  different computation times. Hence, there is a clear need to address the computational bottleneck of repeatedly solving the local subproblems in distributed optimization algorithms. 

In the ADMM framework, the subproblems are solved iteratively by fixing the alternate directions and the dual variables in the each subproblem. Noting that between each iteration, the subproblems solved are similar, this paper  aims to improve the computational speed of the procedure by exploiting the parametric sensitivity of the subproblems. That is, using the optimal solutions computed at previous iterations, we use the parametric sensitivities to estimate how this solution changes when the parameters change. This enables us to cheaply evaluate the subproblems at subsequent iterations, thus reducing the overall computation time. 
 
 \begin{figure}
 	\centering
 	\includegraphics[width=0.65\linewidth]{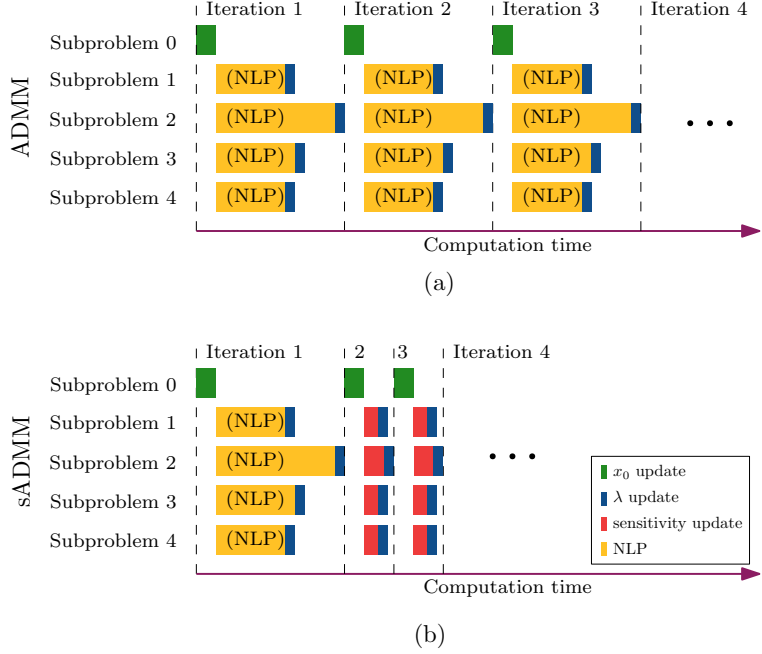}
 	\caption{Schematic representation of the computation time and memory locking in distributed optimization iterations using (a) traditional ADMM, (b) sensitivity-assisted ADMM.}\label{Fig:MemoryLocking}
 \end{figure}

\paragraph*{Related work}
Repeatedly solving the optimization problem  can be computationally expensive, especially if a large number of iterations are required. 
To improve the rate of convergence, algorithms such as ALADIN \cite{houska2016augmented,houska2018ALADIN} and quadratic approximation \cite{wenzel2016optimal}  have been proposed, where the dual variable update step in the master problem uses Newton steps instead of subgradient descent. As such, research efforts have been predominantly 
aimed at addressing the computational bottleneck by  improving the master problem, such that the number of iterations $ k $ between the master and the subproblems is reduced.  Although these are very important developments in the field of distributed optimization, the computation burden of solving the local subproblems still remains.

 To address the computation burden of solving the subproblems, linearized and inexact ADMM methods have been proposed, see for example \cite{ouyang2015accelerated,aybat2017distributed,smith2018cocoa,yang2013linearized,qiao2016linearized}, where instead of solving the exact subproblems, simpler optimization problems are solved at each iteration, which are obtained by linearizing the optimization problem. More specifically, in linearized ADMM, the convex objective function, (and the quadratic penalty terms in case of preconditioned linearized ADMM)  are linearized around the current iterate $x^{k}$, and the subproblems solve a simplified, yet inexact optimization problem at each iteration to compute $x^{k+1}$.

Our proposed approach is fundamentally different from the existing works on linearized ADMM \cite{ouyang2015accelerated,aybat2017distributed,smith2018cocoa,yang2013linearized,qiao2016linearized} in the sense that, our proposed approach is based on the idea that the solution at the current iterate is similar to the solution obtained at the previous iterate, and we use this idea to compute an approximate solution at subsequent iterations. 
In other words, in our proposed approach, the solution to the ADMM subproblems are approximated from one iteration to the next using parametric sensitivity updates. Unlike linearized ADMM, this does not involve linearizing the objective function, but instead the tangential predictor linearizes the solution manifold around the parameter. The linearization is still of the optimality conditions, and thus it can be seen as an approximate second order method, akin to inexact Newton, rather than a first-order one.

\paragraph*{Main contribution}
 To this end, the main contribution of this paper  is a  {s}ensitivity-assisted {A}lternating {D}irections {M}ethod of {M}ultipliers (sADMM) framework to efficiently solve distributed learning problems, where the subproblems are solved approximately using a tangential predictor. The paper  also provides a convergence analysis of the  proposed sensitivity-assisted ADMM scheme and discusses some extensions and variations of the proposed method. The effectiveness of the proposed approach is  demonstrated using  a nonlinear regression example, as well as a classification example using real data sets obtained from the UCI machine learning repository \cite{Dua:2019}.

\input{problem_formulation}

\input{proposed_method}

\input{convergence}

\input{discussion}

\input{experiments}

\section{Conclusion}
This paper  considered distributed learning problems, where the learning problem is decomposed into several smaller subproblems by splitting across the data set, such that each subproblem can be solved by a worker computing node, coordinated by a master problem. To address the computational cost of such problems, we proposed a sensitivity-assisted ADMM algorithm, where the parametric sensitivities of the optimization problems are exploited to cheaply evaluate the approximate solution to the subproblems, thus improving the computational cost of the distributed learning problem. Using several illustrative examples, the proposed approach was shown to provide similar performance as the traditional ADMM algorithm, but at a significantly lower computational costs. 
The proposed sensitivity-assisted ADMM can also be easily extended to constrained optimization problems, under suitable constraint qualifications, and strict complementarity conditions. 
An interesting future research direction would also be to extend the idea of exploiting parametric sensitivities to other distributed optimization frameworks that requires solving the subproblems iteratively, such as in DANE, ALADIN and quadratic approximations to name a few. 

\bibliographystyle{IEEEtran}
\bibliography{mybib}

\end{document}

%% file: problem_formulation.tex
\section{Problem formulation}
Consider a general model fitting problem using the labeled training data set \[  \mathcal{D}:= \{(u_{j},y_{j})\}_{j=1}^{M}  \] where $ u_{j} \in \mathbb{R}^m$ denotes the set of features, and $ y_{j} \in \mathbb{R}^q $ denotes the labels for each sample. The overall feature and label space in the data set $ \mathcal{D} $ is thus denoted as $ \mathbf{u} \in \mathbb{R}^{M\times m} $ and $ \mathbf{y} \in \mathbb{R}^{M\times q} $, respectively. 

The objective is to fit a parametric function  \[ y = f(u,x) \]  parameterized by the vector $ x \in \mathbb{R}^n $ using the labeled training data set $ \mathcal{D} $. A systematic procedure for choosing the parameter vector $ x $ is to solve a numerical optimization with training data $ \mathcal{D} $ input into the optimization problem as,
\begin{equation}\label{Eq:Training}
	x^* = \arg \min_{x} J(x;\mathcal{D})+ h(x)
\end{equation}
where  \[ J(x;\mathcal{D}) := \frac{1}{M}\sum_{j=1}^{M}\ell(f(u_{j},x) - y_{j})^2  \]  is the least squares loss function, and $ h: \mathbb{R}^n \rightarrow \mathbb{R} $ is a regularization function.

\begin{assume}\label{asm:prob}
$J(x;\mathcal{D})$ is smooth but can be nonconvex, and the regularization term $h(x)$ is convex, but can be non-smooth.
\end{assume}
A wide class of regression and classification problems can be put in the form of a general model fitting problem (\ref{Eq:Training}), with $ f(\cdot,\cdot) $ , $ J(\cdot) $ and $ h(\cdot) $ chosen appropriately that satisfies Assumption~\ref{asm:prob}. 

If the learning problem \eqref{Eq:Training} leads to a large-scale optimization problem, then it can be decomposed into smaller subproblems.
The optimization problem can either be decomposed across the training data set $ \mathcal{D}$, or decomposed across the features $ u $ \cite{boyd2011ADMM}. In this paper, we consider the former case when we have a large number of training samples. Here,  the data set can be divided into $ N $ data chunks $ \mathcal{D}_{i} := \{(u_{j},y_{j})\}_{j=1}^{M_{i}}  $ such that \[  M = \sum_{i=1}^{N} M_{i}  \quad \textup{and} \quad  \mathcal{D} = \bigcup_{i=1}^{N} \mathcal{D}_{i} \]
The optimization problem is then given by
\begin{align}\label{Eq:Training2}
	\min_{x} & \; h(x) + \sum_{i=1}^{N} J_i(x;\mathcal{D}_{i})
\end{align}
The shared variable $ x $ couples the different subproblems together. In order to decompose the learning problem, a local copy $ x_{i} \in \mathbb{R}^n $ of the shared variable $ x $ is introduced,  such that the optimization problem \eqref{Eq:Training2} can be written as
\begin{subequations}\label{Eq:consensus}
	\begin{align}
		\min_{x_0,x_{1},\dots,x_{N}}&h(x_{0}) + \sum_{i=1}^{N} J_{i}(x_{i};\mathcal{D}_{i}) \label{Eq:consensus1}\\
		\textup{s.t.} & \quad  x_{i} = x_{0} \quad \forall i = 1,\dots,N \label{Eq:consensus2}
	\end{align}
\end{subequations}
The constraint  \eqref{Eq:consensus2} ensures that the local variables all agree, such that we build a global model in a distributed manner. 
This problem formulation is also known as  the \textit{global variable consensus problem}, which can be solved using ADMM \cite{boyd2011ADMM}.  

The augmented Lagrangian of  \eqref{Eq:consensus} is given by 
\begin{align}\label{Eq:consensusAL}
	\min_{x_0,x_{1},\dots,x_{N}}&  	\mathcal{L}(\{x_i\},x_0,\{\lambda_{i}\})  :=  h(x_{0})+ \sum_{i=1}^{N} J(x_{i};\mathcal{D}_{i}) \nonumber \\ \quad & + \sum_{i=1}^{N} \lambda_{i}^{\mathsf{T}} (x_{i} - x_{0})  + \sum_{i=1}^{N} \frac{\rho}{2}  \|x_{i} - x_{0}\|^2
\end{align}
where $ \lambda_{i} \in \mathbb{R}^n$ is the Lagrange multiplier that corresponds to the consensus equality constraint \eqref{Eq:consensus2}. Unless otherwise explicitly denoted with a subscript $\|\cdot\|$ denotes the Euclidean norm, and the set of variables $ \{ (\cdot)_{i}\} $ by default denotes $ \{ (\cdot)_{i}\}_{i=1}^{N} $.

It can be seen that the learning problem  \eqref{Eq:consensusAL} is  additively separable except for the quadratic penalty terms $ \|x_{i} - x_{0}\|^2  $. 
Therefore the subproblems $ i= 1,\dots,N $ are solved by fixing $ x_{0} $ and $ \lambda_{i} $, and the subproblem $ i=0 $ is solved by fixing $ x_{i} $ and $ \lambda_{i} $ for all $ i = 1,\dots,N $ in an alternating directions fashion. The ADMM method then consists of the iterations \cite{boyd2011ADMM}, 
\begin{subequations}\label{Eq:ADMM}
	\begin{align}
			x_0^{k+1} & = \arg \min_{x_0}  \mathcal{L}(\{x_i^k\},x_0,\{\lambda_{i}^k\}):=   h(x_{0})
			+  \sum_{i=1}^{N} \lambda_i^{k\mathsf{T}} (x_i^{k} - x_0)   + \sum_{i=1}^{N}   \frac{\rho}{2} \left\| x_i^{k} - x_0 \right\|^2   \label{Eq:ADMM_x0}\\
		x_i^{k+1} & = \arg \min_{x_i}   \mathcal{L}_{i}(x_{i},x_{0}^{k+1},\lambda_{i}^k):= J_{i}(x_{i};\mathcal{D}_{i}) 
		+ \lambda_i^{k\mathsf{T}} (x_i - x_0^{k+1})  + \frac{\rho}{2} \left\| x_i - x_0^{k+1} \right\|^2  \label{Eq:ADMM_xi} \\
		& \qquad\qquad\qquad\qquad\qquad \quad\quad \forall i = 1,\dots,N\nonumber\\
		\lambda_i^{k+1} &= \lambda_i^k + \rho(x_i^{k+1} - x_0^{k+1}) \qquad \forall i = 1,\dots,N\label{Eq:ADMM_master}
	\end{align}
\end{subequations}
starting with the initial guess $\{x_i^0\},\{\lambda_{i}^0\}$ at $k=0$, where $ k $ is the ADMM iteration number. For the sake of notational simplicity, we denote  the transpose of $\lambda_i^{k}$ as  $\lambda_i^{k\mathsf{T}}$ instead of $(\lambda_i^{k})^\mathsf{T}$.

We consider the case where $ J_i(x_i;\mathcal{D}_i) $ is smooth, but may be nonlinear and non-convex (cf. Assumption~\ref{asm:prob}). Hence  \eqref{Eq:ADMM_xi} is a nonlinear programming (NLP) problem. The optimization problem  \eqref{Eq:ADMM_x0} is often called the \textit{central collector}, since $ x_{0} $ is used by all the subproblems in \eqref{Eq:ADMM_xi}. The final step \eqref{Eq:ADMM_master} is a dual-ascent step that updates the dual variables $\lambda_i$. The steps \eqref{Eq:ADMM_xi} and \eqref{Eq:ADMM_master} are carried out in parallel by the different worker nodes for each $i= 1,\dots,N$.
At each time,  \eqref{Eq:ADMM_x0},\eqref{Eq:ADMM_xi}, and \eqref{Eq:ADMM_master} are iteratively solved until some stopping criteria is met \cite{eckstein1992douglas,boyd2011ADMM}. 

\begin{remark}[Closed-form solution for \eqref{Eq:ADMM_x0}]\label{rem:closedFormSol}
In the case of $\ell_1$ regularization, i.e. $ h(x) = \omega\|x\|_1 $, subproblem \eqref{Eq:ADMM_x0} is non-differentiable, but one can use subdifferential calculus to compute a closed-form solution, which is given by a soft-thresholding operator, also known as a shrinkage operator as shown below,
\begin{equation}\label{Eq:ADMM_x0_explicit_l1}
   x_0^{k+1} = \mathcal{S}_{\frac{\omega}{N\rho}}\left( \frac{1}{N}\sum_{i=1}^{N} \left[ x_i^{k}  + \frac{\lambda_i^k}{\rho} \right]  \right)
\end{equation}
where the soft-thresholding/shrinkage operator is given by
\begin{equation*}
\mathcal{S}_{\kappa}\left( a \right) = \max[0,a-\kappa] - \max[0,-a-\kappa]
\end{equation*}
See \cite[Sec. 6]{boyd2011ADMM} for further details.

In the case of $\ell_2$ regularization, i.e. $ h(x) = \omega\|x\|_2^2 $, subproblem \eqref{Eq:ADMM_x0} has the closed form solution
\begin{equation}\label{Eq:ADMM_x0_explicit_l2}
	x_0^{k+1} = \frac{1}{2\omega/\rho +N}\sum_{i=1}^{N} \left[ x_i^{k}  + \frac{\lambda_i^k}{\rho} \right]  
\end{equation}

If no regularization terms are used, i.e. $ h(x) = 0 $, then the closed-form solution is obtained by setting $ \omega = 0 $ in \eqref{Eq:ADMM_x0_explicit_l1} or \eqref{Eq:ADMM_x0_explicit_l2}.
\end{remark}

To this end, each subproblem solves for the model parameters $ x_{i} $ in parallel using only a subset of the training samples $ \mathcal{D}_{i} $, and the central coordinator ensures that the parameters $ x_{i} $ computed by the different subproblems converge to the same value,  that is, the different subproblems collaborate to develop a global model. 
	
For primal feasiblity of the ADMM \eqref{Eq:ADMM}  
we need $ x_{i}^* = x_{0}^*$ for all $ i=1,\dots,N $, and for dual feasibility we need
\begin{align*}
\partial_{x0} h (x_{0}^*)  - \sum_i \lambda_i^* &\ni 0\\
\nabla_{x_i}f_{i}(x_{i}^*) + \lambda_i^* & = 0
\end{align*}
By definition, at each iteration we have dual feasibility of the subproblems \eqref{Eq:ADMM_xi} $ \forall i= 1,\dots,N $
\begin{align*}
	\nabla J_i( x_i^{k+1})+\lambda_i^{k}+\rho (x_i^{k+1}-x_0^{k+1}) &= 0\\
	\nabla J_i( x_i^{k+1})+\lambda_i^{k+1}&=0
\end{align*}
Looking at the optimality condition of \eqref{Eq:ADMM_x0} at iteration $ k+1 $
\begin{align*}
	\partial_{x0} h (x_{0}^{k+1})  - \sum_{i=1}^{N} \left[\lambda_i^k + \rho (x_{i}^k-x_{0}^{k+1}) \right] &\ni 0\\
	\partial_{x0} h (x_{0}^{k+1})  - \sum_{i=1}^{N} \lambda_i^{k+1}  + \sum_{i=1}^{N} \rho (x_{i}^{k+1}-x_{i}^{k}) &\ni 0\\
\end{align*}
Hence the primal and dual residual at iteration $ k+1 $ is given by
\begin{equation*}
	r^{k+1} = \begin{bmatrix}
	x_{1}^{k+1} - x_{0}^{k+1}\\
	\vdots \\
	x_{N}^{k+1} - x_{0}^{k+1} 
\end{bmatrix}, \qquad	s^{k+1} = \sum_{i=1}^{N} \rho (x_{i}^{k+1}-x_{i}^{k})
\end{equation*}
respectively.


%% file: proposed_method.tex
\section{Sensitivity-assisted ADMM}\label{sec:sADMM}
It can be seen from \eqref{Eq:ADMM_xi} that the subproblems $ i = 1, \dots,N $ are solved at iteration $k+1$ by fixing $ x_{0}^{k+1} $ and $ \lambda_{i}^k $. Once $ x_{0} $ and $ \lambda_{i} $ are updated in \eqref{Eq:ADMM_x0} and \eqref{Eq:ADMM_master} respectively, the $ N $ subproblems \eqref{Eq:ADMM_xi} are solved again with the updated value of $ x_{0} $ and $ \lambda_{i} $. This is repeated until some stopping criteria is met \cite{boyd2011ADMM}. Iteratively solving the optimization subproblems can be time consuming and computationally expensive. 

In order to address this issue, we now propose a sensitivity-assisted ADMM. The underlying idea of the proposed sensitivity-assisted ADMM is as follows. Since the only difference between two consecutive iterations of the subproblems \eqref{Eq:ADMM_xi} is the value of $ x_{0}^{k+1} $ and $ \lambda_{i}^k $, the subproblems can  equivalently be written as a parametric optimization problem 
\begin{align}\label{Eq:par_subproblem}
x_i^*(p_i^{k+1}) & = \arg \min_{x_i}   \mathcal{L}_{i}(x_i,p_i^{k+1}), \quad \forall i=1,\dots,N
\end{align}
where $p_i^{k+1} = \begin{bmatrix}
	x_0^{k+1}\\
	\lambda_i^k
\end{bmatrix}$ denotes the set of parameters that are updated at each iteration. We also introduce the notation $x_i^*(p_i^{k+1}):=x_i^{k+1}$  to explicitly show the dependence of the optimal solution of \eqref{Eq:ADMM_xi} on the parameter $ p_{i}^{k+1} $. 
Once the solution to  subproblems \eqref{Eq:par_subproblem} are evaluated for a given parameter  $p_i^k = \begin{bmatrix}
	x_0^{k}\\
	\lambda_i^{k-1}
\end{bmatrix}$ by solving the nonlinear programming problem exactly,  parametric sensitivity can be used to cheaply evaluate how the optimal solution $ x_{i}^*(p_{i}^k) $ changes at the subsequent iteration when the parameter  $p_i^{k+1} = \begin{bmatrix}
	x_0^{k+1}\\
	\lambda_i^{k}
\end{bmatrix}$ is updated. 

The KKT condition for the unconstrained subproblem is given by
\begin{equation}\label{Eq:KKT}
\nabla_{x_i}\mathcal{L}_{i}(x_i,p_i^k) = 0
\end{equation}
and $ x_i^*(p_i^{k}) $ is called a KKT-point of \eqref{Eq:par_subproblem} that satisfies \eqref{Eq:KKT} for any $p_i^k$ for all $i= 1,\dots,N$.

\begin{assume} \label{asm:KKT}
	The following holds:
	\begin{enumerate}
		\item $ \mathcal{L}_{i}(x_i,p_i)$ is smooth and twice differentiable in $p_i$ and $x_i$ in the neighborhood of $x_i(p_i^k)$.
		\item $x_i^*(p_i^k)$ is a KKT point of \eqref{Eq:par_subproblem}.
		\item Strong second order sufficient conditions (SSOSC) hold at any KKT point $x_i^*(p_i^k)$, i.e.
		\[ d^T \nabla_{x_i,x_i}^2 \mathcal{L}_{i}(x_i,p_i^k)d >0 \quad \forall d \neq 0\]
	\end{enumerate}
\end{assume}

\begin{theorem}\label{thm:sensitivity}
Given Assumptions~\ref{asm:KKT}, the following holds for all $i=1,\dots,N$:
\begin{enumerate}
	\item $x_i^*(p_i^k)$ is a local minimizer of $ \mathcal{L}_{i}(x_i,p_i^k)$
	\item For $p_i^{k+1}$ in the neighborhood of $p_i^k$, there exists a unique, continuous and differentiable vector $x_i^*(p_i^{k+1})$ which is a local minimizer of $ \mathcal{L}_{i}(x_i,p_i^{k+1})$
	\item There exists $\alpha>0$ such that $\| x_i^*(p_i^{k+1}) - x_i^*(p_i^{k})\| \leq \alpha \| p_i^{k+1} -p_i^{k}\| $.
	\item There exists $L_{J} > 0$ such that $\| \mathcal{L}_{i}^*(p_{i}^{k+1}) - \mathcal{L}_{i}^*(p_{i}^{k})\| \leq L_{J} \| p_{i}^{k+1} -p_{i}^{k}\| $.
\end{enumerate}
\end{theorem}
\begin{proof}
	See \cite{fiacco1976sensitivity}
\end{proof}

Therefore, we can apply implicit function theorem on the KKT conditions, and compute
\begin{align*}
&\nabla_{x_i,x_i}^2 \mathcal{L}_{i}(x_i^*(p_i^k),p_i^k) \nabla_{p_{i}} x_{i}^*(p_i)
=- \nabla_{x_i,p_{i}}^2 \mathcal{L}_{i}(x_i^*(p_i^k),p_i^k)
\end{align*}

Linearization of the KKT condition around the nominal solution  $x_i^*(p_i^k)$ gives,
\begin{align*}
0=&\underbrace{\nabla_{x_i,x_i}^2 \mathcal{L}_{i}(x_i^*(p_i^k),p_i^k) }_{:= \mathcal{M}_i}( x_i^*(p_i^{k+1}) - x_i^*(p_i^k)) \nonumber \\
&+\underbrace{\nabla_{x_i,p_i}^2 \mathcal{L}_{i}(x_i^*(p_i^k),p_i^k)}_{:=\mathcal{N}_i} (p_i^{k+1} - p_i^k) + \mathcal{O}(\|p_i^{k+1} - p_i^k\|^2)
\end{align*}

The first-order \enquote{tangential predictor} estimates for the updated parameters $p_i^{k+1} $  can then be written as,
\begin{align}\label{Eq:sensitivity}
\tilde{x}_i^*(p_i^{k+1})  = {x}_i^*(p_i^{k}) - \mathcal{M}_i^{-1}\mathcal{N}_i&(p_i^{k+1} - p_i^k)  \\
&\forall i = 1,\dots,N \nonumber 
\end{align}
where $\tilde{x}_i^*(p_i^{k+1}) $ is an approximation of the solution ${x}_i^*(p_i^{k+1}) $. 
Note that $ \mathcal{M}_{i} $ is also the KKT matrix that is used when solving the NLP \eqref{Eq:ADMM_xi} exactly. 

In the ADMM iterations, if we have the optimal solution $ x_{i}^{k} $ for the subproblems \eqref{Eq:ADMM_xi} at iteration $ k $, then the optimal solution at iteration $ k+1 $ can be cheaply approximated using \eqref{Eq:sensitivity},
 and due to the continuity and the differentiability of the solution vector, the difference between the approximate solution and the true solution is bounded, that is,  \begin{equation}\label{Eq:SolnBound}
 	\| \tilde{x}_i^{k+1}  - {x}_i^{k+1} \| \leq L_x \| p_i^{k+1} -p_i^{k}\|^2 
 \end{equation} for some positive constant $ L_x $.  Consequently, the optimality condition of the subproblems \eqref{Eq:ADMM_xi}  evaluated at $  \tilde{x}_{i}^{k+1} $, given by 
\begin{align}\label{eq:sensopt}
\nabla_{x_i}\mathcal{L}_{i}(\tilde{x}_{i}^{k+1},p_{i}^{k+1}) := 	\nabla J_i(\tilde x_i^{k+1})+\lambda_i^{k}+\rho (\tilde x_i^{k+1}-x_0^{k+1})  = \epsilon_{i}^{k+1}
\end{align} for all $ i = 1,\dots,N $ may not be equal to zero.

\begin{algorithm}[t]
	\caption{Sensitivity-assisted ADMM (sADMM)}
	\begin{algorithmic}[1]
		\Require  $\{x_i^0\},\{\lambda_i^0\}$, $ \rho$, $ D $, $ R $
		\algrulehor
		\For k = 0,1,2,\dots
		\State $ x_0^{k+1} \leftarrow \mathcal{S}_{\frac{\omega}{N\rho}}\left(\frac{1}{N} \sum_{i=1}^{N} \left[ \tilde{x}_i^{k}  + \frac{\lambda_i^k}{\rho} \right] \right)$ \Comment cf. Remark~\ref{rem:closedFormSol}
		\For {$ i= 1,\dots,N $ (in parallel)} 
		\State $ p_{i}^{k+1} \leftarrow \left[x_0^{k+1},\lambda_i^k\right]^\mathsf{T} $
		\If  {$ \| \tilde{x}_i^{k} -x_{0}^{k}\|> R $}\Comment Solve NLP
		\State $\tilde{x}_i^{k+1} \leftarrow \arg \min_{x_i}   \mathcal{L}_{i}(x_i,p_{i}^k)$ 
		\Else \Comment use sensitivity update
		\State $ \mathcal{M}_{i} \leftarrow  \nabla_{x_i,x_i}^2 \mathcal{L}_{i}(\tilde{x}_i^{k},p_i^{k}) $
		\State $ \mathcal{N}_{i}  \leftarrow  \nabla_{x_i,p_i} \mathcal{L}_{i}(\tilde{x}_i^{k},p_i^{k}) $
		\State $\tilde{x}_i^{k+1} \leftarrow \tilde{x}_i^{k}- \mathcal{M}_{i}^{-1}\mathcal{N}_{i} (p_{i}^{k+1}-p_{i}^{k})$ 
			\While {$ \|\nabla_{x_i}\mathcal{L}_{i}(\tilde{x}_{i}^{k+1},p_{i}^{k+1}) \| > D $} \Comment{Corrector steps}
			\State $ \mathcal{M}_{i}  \leftarrow \nabla_{x_i,x_i}^2 \mathcal{L}_{i}(\tilde{x}_i^{k+1},p_i^{k+1}) $
		\State $\Delta x_{i}^c \leftarrow - \mathcal{M}_{i}^{-1} \nabla_{x_i}\mathcal{L}_{i}(\tilde{x}_{i}^{k+1},p_{i}^{k+1})$ 
		\State $ \tilde{x}_{i}^{k+1} \leftarrow \tilde{x}_{i}^{k+1} + \Delta x_{i}^c $
		\EndWhile
		\EndIf
		
		\State $  \lambda_i^{k+1}  \leftarrow \lambda_i^k + \rho(\tilde{x}_i^{k+1} - x_0^{k+1})$ \Comment Dual update 
		\EndFor
		\EndFor
		\algrulehor
		\Ensure $\{x_i^{k+1}\},x_{0}^{k+1}\{\lambda_i^{k+1}\}$
	\end{algorithmic}
\end{algorithm}

It can be seen that this approximation error depends on the problem Lipschitz functions as affecting how much the problem changes iteration to iteration, and thus the accuracy of a tangential predictor. In practice, we can add additional corrector steps to obtain an optimality residual as desired a priori, i.e., the optimality residual  can be bounded  by  $\| \epsilon_i^{k+1} \| \leq  D $ for an user-defined  $ D $.  In particular,  the sensitivity update involves one or multiple (along a homotopy) tangential steps (a linear system solve) and (possibly) a corrector (amounting to another linear system solve, in this case). 
That is, the corrector-step is given by 
\begin{align}
	\Delta x_{i}^c &= - \mathcal{M}_{i}^{-1} \nabla_{x_i}\mathcal{L}_{i}(\tilde{x}_{i}^{k+1},p_{i}^{k+1}) \\
	\tilde{x}_{i}^{k+1} & = \tilde{x}_{i}^{k+1} + \Delta x_{i}^c
\end{align}
with $ \mathcal{M}_{i} :=  \nabla_{x_i,x_i}^2 \mathcal{L}_{i}(\tilde{x}_i^{k+1},p_i^{k+1})$. 

The number of steps, the presence of, and the tightness of optimality asked of the corrector can be set by the user, and this represents a trade off between the accuracy (and thus the size of $ D $) and the computational expense. Since the corrector steps also only involves solution to a system of linear equations, this is computationally much easier than solving the NLP exactly.  Simply put, by choosing the size of $ D $, the user  has a direct control over the trade-off between accuracy and computational cost. 

In principle, one can approximate the subproblem solutions after iteration $ k>1 $. In this paper, we switch from solving the subproblems approximately if the primal residual $ \|\tilde{x}_{i}^k - x_{0}^k\| \le R$ for some user defined $ R $. Note that the change in the $ \lambda_{i} $ from one iteration to the next depends on the primal residual $ \|\tilde{x}_{i}^k - x_{0}^k\| $.  Therefore, using this as the criteria for switching helps one to control that the parametric variations are \enquote{sufficiently} small.  


To this end, the proposed sADMM algorithm requires two user-defined parameters, namely the acceptable primal residual $ R  $ that tells  when to approximate the suproblems using  parametric sensitivities, and maximum allowed error bound on the optimality condition $ D $, which decides the number of the corrector steps taken by the algorithm. The proposed sensitivity-assisted ADMM approach is summarized in Algorithm 1.

%% file: convergence.tex
\section{Convergence properties}\label{sec:convergence}
In this section, we analyze the convergence properties of the proposed sensitivity-assisted ADMM approach.

\begin{defn}[Strong Convexity]
	Any function $ \mathcal{L}: \mathbb{R}^n \rightarrow \mathbb{R} $ is said to be strongly convex with modulus $ \gamma $ if for any $ a,b \in \mathbb{R}^n $
	\[ \mathcal{L}(a) - \mathcal{L}(b) \leq \nabla \mathcal{L}(a)^{\mathsf{T}}(a-b) - \frac{\gamma}{2}\|a-b\|^2 \]
\end{defn}
We make the following Assumption in regards to the problem,
\begin{assume}\label{as:as}
	\begin{enumerate}
		\item Every $J_i$ is Lipschitz continuously differentiable with constant $L_i$
		\[  \left\|\nabla J_i(a)-\nabla J_i(b)\right\| \le L_i \|a-b\|, \quad \forall a,b \in \mathbb{R}^n\]
		\item The penalty parameter $\rho$ satisfies $\rho \gamma_{i}(\rho)\ge 8L_i^2$, and so each subproblem  $ \mathcal{L}_{i} $ is strongly
		convex with modulus $\gamma(\rho)$
		\item $J$ is bounded from below
	\end{enumerate}
\end{assume}
Note that we allow for the functions $\{J_i\}$ to be nonconvex in general.

By using sensitivity updates to update the vectors $\tilde x_i$, each problem is solved inexactly, however with
a fixed tracking bound on the optimality residual $ \|\epsilon_i^{k+1}\|\leq D $ as mentioned earlier. 
Now we have,
\begin{lemma}\label{lem:lem1}
	(Like~\cite[Lemma 2.1]{hong2016convergence})
	It holds that,
	\[
	\|\lambda_i^{k+1}-\lambda_i^k\|^2 \le 2L_i^2\|\tilde{x}_i^{k+1}-\tilde{x}^k_i\|^2+8D^2
	\]
\end{lemma}
\begin{proof}
	Taking the approximate optimality condition~\eqref{eq:sensopt} and combine this with the dual update~\eqref{Eq:ADMM_master}
	to deduce,
	\[
	\nabla J_i(\tilde{x}_i^{k+1}) = -\lambda_i^{k+1}+\epsilon_i^{k+1}
	\]
	and thus,
	\begin{align*}
		\|\lambda_i^{k+1}-\lambda_i^k\| =& \left\|\nabla J_i(\tilde{x}_i^{k+1})-\nabla J_i(\tilde{x}_i^{k})\right\| 
+\|\epsilon_i^{k+1}-\epsilon_i^{k}\| \le L_i \|\tilde{x}_i^{k+1}-\tilde{x}^k_i\|+2D
	\end{align*}

	and the final result follows from the fact that $(a+b)^2\le 2a^2+2b^2$.
\end{proof}

\begin{lemma}\label{lem:lem2}
	(Like~\cite[Lemma 2.2]{hong2016convergence}) 
	\begin{align*}
		&	\mathcal{L}(\{\tilde{x}_i^{k+1}\},x_0^{k+1},\{\lambda_{i}^{k+1}\})-\mathcal{L}(\{\tilde{x}_i^{k}\},x_0^{k},\{\lambda_{i}^{k}\})\\
		\quad & \le \sum\limits_i \left(\frac{2L_i^2}{\rho}-\frac{\gamma_{i}(\rho)}{4}\right)\|\tilde{x}^{k+1}_i-\tilde{x}^k_i\|^2 
		- \frac{\gamma}{2}\|x_0^{k+1}-x_0^k\|+\frac{8ND^2}{\rho_m}
	\end{align*}
	where $\rho_m=\min\{\rho,\{\gamma_i(\rho)\}\}$.
\end{lemma}
\begin{proof}
	We have that,
	\begin{align}\label{eq:lem2twopart}
	&\mathcal{L}(\{\tilde{x}_i^{k+1}\},x_0^{k+1},\{\lambda_{i}^{k+1}\})-\mathcal{L}(\{\tilde{x}_i^{k}\},x_0^{k},\{\lambda_{i}^{k}\})  \nonumber \\
	&= \underbrace{\mathcal{L}(\{\tilde{x}_i^{k+1}\},x_0^{k+1},\{\lambda_{i}^{k+1}\})-\mathcal{L}(\{\tilde{x}_i^{k+1}\},x_0^{k+1},\{\lambda_{i}^{k}\})}_{=:\mathcal{A}} \nonumber \\ 
	& \qquad+\underbrace{\mathcal{L}(\{\tilde{x}_i^{k+1}\},x_0^{k+1},\{\lambda_{i}^{k}\})-\mathcal{L}(\{\tilde{x}_i^{k}\},x_0^{k},\{\lambda_{i}^{k}\})}_{=:\mathcal{B}}
	\end{align}
	Now $ \mathcal{A} $ satisfies,
	\begin{align}\label{eq:firsttermlagdesc}
	\mathcal{A}&=\mathcal{L}(\{\tilde{x}_i^{k+1}\},x_0^{k+1},\{\lambda_{i}^{k+1}\})-\mathcal{L}(\{\tilde{x}_i^{k+1}\},x_0^{k+1},\{\lambda_{i}^{k}\})\nonumber\\
	&= \sum\limits_{i=1}^N \left\langle \lambda^{k+1}_i-\lambda^k_i,\tilde{x}^{k+1}_i-x_0^{k+1}\right\rangle \nonumber\\
	&= \sum\limits_{i=1}^N \frac{1}{\rho}\left\|\lambda^{k+1}_i-\lambda^k_i\right\|^2 \nonumber\\
	& \leq \sum\limits_{i=1}^N \frac{2L_{i}^2}{\rho}\left\|\tilde{x}^{k+1}_i-\tilde{x}^k_i\right\|^2 + \frac{8ND^2}{\rho}
	\end{align}
	The term $ \mathcal{B} $ can be bounded as follows,
\begin{align*}
\mathcal{B}	&=\mathcal{L}(\{\tilde{x}_i^{k+1}\},x_0^{k+1},\{\lambda_{i}^{k}\})-\mathcal{L}(\{\tilde{x}_i^{k}\},x_0^{k},\{\lambda_{i}^{k}\}) \\
& = \mathcal{L}(\{\tilde{x}_i^{k+1}\},x_0^{k+1},\{\lambda_{i}^{k}\})-\mathcal{L}(\{\tilde{x}_i^{k}\},x_0^{k+1},\{\lambda_{i}^{k}\})
\\
& \qquad+ \mathcal{L}(\{\tilde{x}_i^{k}\},x_0^{k+1},\{\lambda_{i}^{k}\})-\mathcal{L}(\{\tilde{x}_i^{k}\},x_0^{k},\{\lambda_{i}^{k}\}) 
\end{align*}
Since the augmented Lagrangian is $ \gamma $-strongly convex,
\begin{align*}
\mathcal{B}& \le \sum\limits_{i=1}^N  \left\langle \nabla_{x_i} \mathcal{L}(\left\{\tilde{x}_i^{k+1}\right\},x_0^{k+1},\{\lambda_{i}^{k}\}),\tilde{x}_i^{k+1}-\tilde{x}_i^k\right\rangle \\
& \qquad -\sum\limits_{i=1}^N \frac{\gamma_i(\rho)}{2} \|\tilde{x}_i^{k+1}-\tilde{x}_i^k\|^2 
+\left\langle \partial_{x0} \mathcal{L}\left(\{\tilde{x}_{i}^k\},x_{0}^{k+1},\{\lambda_i^k\}\right), x_{0}^{k+1} - x_{0}^k\right\rangle \\
&\qquad -\frac{\gamma}{2}\left\|x_0^{k+1}-x_0^k\right\|^2 \\ 
\end{align*}
where $ \partial_{x0}\mathcal{L}\left(\{\tilde{x}_{i}\},x_{0}^{k+1},\{\lambda_i^k\}\right)$ denotes a subgradient. Using the approximate optimality condition \eqref{eq:sensopt} and the optimality condition of \eqref{Eq:ADMM_x0} we have,
\begin{align*}
\mathcal{B} &  \le \sum\limits_{i=1}^N\left(\|\tilde x_i^{k+1}- \tilde x_i^k\|\|\epsilon^k_i\|-\frac{\gamma_i(\rho)}{2} \|\tilde{x}_i^{k+1}-\tilde{x}_i^k\|^2\right)  
-\frac{\gamma}{2} \|x_0^{k+1}-x_0^k\|^2 \\ 
&  \le \sum\limits_{i=1}^N\left(D \|\tilde x_i^{k+1}-\tilde x_i^k\|-\frac{\gamma_i(\rho)}{2} \|\tilde{x}_i^{k+1}-\tilde{x}_i^k\|^2\right) 
-\frac{\gamma}{2} \|x_0^{k+1}-x_0^k\|^2 \\ 
&\le \sum\limits_{i=1}^N\left(\frac{D^2}{\gamma_i(\rho)}-\frac{\gamma_i(\rho)}{4} \|\tilde{x}_i^{k+1}-\tilde{x}_i^k\|^2\right) -\frac{\gamma}{2} \|x_0^{k+1}-x_0^k\|^2
\end{align*}
The last inequality comes from Young's inequality \[ ab \le \frac{a^2}{2\varepsilon} + \frac{\varepsilon b^2}{2}  \] 
with $ a := D $, $ b:=\|\tilde{x}_{i}^{k+1} - \tilde{x}_{i}^k\| $ and $ \varepsilon:= \gamma_{i}(\rho)/2 $.
	Now, using~\eqref{eq:firsttermlagdesc} with this last estimate yields the result.
	
\end{proof}


\begin{lemma}\label{lem:lem3}
	(Like~\cite[Lemma 2.3]{hong2016convergence})
	\[
	\lim\limits_{k\to\infty} \mathcal{L}(\{\tilde{x}_i^k\},x_0^k,\{\lambda_{i}^{k}\}) \ge J_{m}-DR
	\]
\end{lemma}
\begin{proof}
	Now,
\begin{align*}
&\mathcal{L}(\{\tilde{x}_i^k\},x_0^k,\{\lambda_{i}^{k}\}) \\
&  = h(x_{0}^k) + \sum\limits_{i=1}^N \left(J_i(\tilde{x}_i^k)+\left\langle \lambda^k_i ,\tilde{x}_i^k-x_0^k\right\rangle+\frac{\rho}{2}\left\|\tilde{x}_i^k-x_0^k\right\|^2\right) \\ 
& \quad \ge h(x_{0}^k) + \sum\limits_{i=1}^N \left(J_i(\tilde{x}_i^k)+\left\langle \nabla J_i(\tilde{x}^k_i) ,x_0^k-\tilde{x}_i^k\right\rangle 
 -D\|x_0^k-\tilde{x}_i^k\|+\frac{\rho}{2}\left\|\tilde{x}_i^k-x_0^k\right\|^2\right) \\ 
\end{align*}
where we used Lemma~\ref{lem:lem1} for the  inequality above. Since the subproblems are solved approximately only when the primal residual $ \|\tilde{x}_{i}^k - x_{0}^k\| \le R $, we have
\begin{align*}
&\ge  h(x_{0}^k) + \sum\limits_{i=1}^N \left( J_i(x_0^k)-D\|x_0^k-\tilde{x}_i^k\|\right) \ge J_m-DR
\end{align*}
	
\end{proof}

\begin{theorem}\label{th:conv}
	There exists a $\tilde D$ proportional to $D^2$ such that,
	\[
	\limsup\limits_{k\to\infty}\|x_0^{k+1}-x_0^k\|\le \tilde{D}
	\]
	\[
	\limsup\limits_{k\to\infty}\|\tilde{x}_i^{k+1}-\tilde{x}_i^k\|\le \tilde{D}
	\]
	and any limit point of the sequence satisfies
	\[
	\|\nabla J_i(x_i^*)+\lambda^*_i\|^2\le \frac{D^2}{2L_{i}},\,\, \|x_i^*-x_0^*\|\le \frac{2L_{i}^2\tilde{D}+8D^2}{\rho}
	\]
\end{theorem}
\begin{proof}
	Combining Lemmas~\ref{lem:lem2} and~\ref{lem:lem3} and obtaining a telescoping sum on $\mathcal{L}$, we can see that it must hold that,
	\[
	\sum\limits_i \left(-\frac{2L_i^2}{\rho}+\frac{\gamma_{i}(\rho)}{4}\right)\|x^{k+1}_i-x^k_i\|^2+\frac{\gamma}{2}\|x_0^{k+1}-x_0^k\|\le\frac{8ND^2}{\rho_m}
	\]
	and the first result follows from the assumption that $\left(-\frac{2L_i^2}{\rho}+\frac{\gamma(\rho)}{4}\right)> 0$. 
	
	The first limit condition follows immediately from the approximate optimality conditions of the subproblem. The second follows
	from Lemma~\ref{lem:lem1} implying that 
	\[
	\limsup\limits_{k\to\infty}\|\lambda_i^{k+1}-\lambda_i^k\|\le 2L_i^2\tilde{D}+8D^2
	\]
	and the update rule for $\lambda_i^k$.
\end{proof}

%% file: discussion.tex
\section{Implementation aspects}
\subsection{sADMM with additional local constraints}\label{sec:constrained-sADMM}
For the sake of clarity in the presentation, we did not explicitly consider local constraints in Section~\ref{sec:sADMM}. 
If the local subproblems \eqref{Eq:ADMM_xi} also have additional  constraints $ g_{i}(x_{i})\le 0 $, then the KKT conditions would be given by 
\begin{equation*}\label{Eq:KKT2}
	\varphi_{i}(x_{i}(p_{i}),p_{i}) := \begin{bmatrix}
		\nabla_{x_i}\mathcal{L}_{i}(x_i,p_i) + \mu_i^{\mathsf{T}}\nabla_{x_i}g_{i_{\mathbb{A}}}(x_{i})\\
		g_{i_{\mathbb{A}}} (x_{i})
	\end{bmatrix} = 0
\end{equation*}
where $ g_{i_{\mathbb{A}}}(\cdot) \subseteq g_{i}(\cdot) $ denotes the set of active constraints, and $ \mu_i $ represents its corresponding Lagrange multiplier.  Under further assumptions of   linear independence constraint qualification (LICQ) and strict complementarity,  the tangential predictor is given by the linearizing the primal-dual solution vector $ s_{i}^*(p_{i}) := [x_{i}^*(p_{i}), \mu_i(p_{i})]^{\mathsf{T}} $ such that
\[ {s}_{i}^*(p_{i}^{k+1})  \approx {s}_{i}^*(p_{i}^{k})  - \left[\frac{\partial \varphi_{i}}{\partial s_{i}}\right]^{-1}\frac{\partial \varphi_{i}}{\partial p_{i}} (p_{i}^{k+1} - p_{i}^k)\]
If the set of active constraint changes from $ p_{i}^{k} $ to $ p_{i}^{k+1} $, then the solution manifold $ s_{i}^*(p_{i}) $ has non-smooth kinks, which are not captured by the tangential predictor. 
Notice that the tangential predictor $ \eqref{Eq:sensitivity} $ corresponds to the stationarity condition of a particular QP, where the solution of this QP was shown to be the directional derivative of the optimal solution manifold $ s_{i}^*(p_{i}) $, which capture the non-smooth kinks in the solution manifold \cite{bonnans1998optimization,levy2001solution}. This is known as the predictor QP, which in this case reads as:
\begin{align}\label{Eq:PredictorQP}
	\min_{\Delta x_{i}}\; &\frac{1}{2} \Delta x_{i}^{\mathsf{T}} \mathcal{M}_{i} \Delta x_{i} + \Delta x_{i}^{\mathsf{T}} \mathcal{N}_{i}(p_{i}^{k+1}-p_{i}^k) \\
	\text{s.t.}\;  & g_{i} + \nabla_{x_i}g_{i}^{\mathsf{T}}\Delta x_{i} + \nabla_{p_i}g_{i}^{\mathsf{T}}(p_{i}^{k+1}-p_{i}^k) \le 0 \nonumber
\end{align}
In addition, a corrector term $ \Delta x_{i}^{\mathsf{T}} \nabla_{x_i}\mathcal{L}_{i}(x_{i},p_{i}) $ can be added to the cost function in \eqref{Eq:PredictorQP} to further improve the approximation accuracy leading to  \textit{predictor-corrector QP}.
The approximate solution to the subproblems are then computed as $ \tilde{x}_{i}^{k+1} = \tilde{x}_{i}^{k} + \Delta x_{i}^* $ in line number 8 in Algorithm~1.   Note that the convergence properties presented in Section~\ref{sec:convergence} also follow in this case. 

Alternatively, one can also solve the NLP exactly when the set of active constraint changes, and use the sensitivity updates using the newly computed solution, in a similar fashion as the SsADMM described above.

\subsection{Stochastic sADMM}\label{sec:SsADMM}
Another alternative implementation is to solve the subproblems \eqref{Eq:ADMM_xi} exactly with probability $ (\delta)^{k} $ at iteration $ k  $ for some user-defined $ \delta \in (0,1) $. By using $ \delta $ to the power of the iteration number $ k $, the probability that the subproblems are solved exactly reduces exponentially with the iteration number. We term this the \textit{stochastic sensitivity-assisted ADMM} (SsADMM). If $ \delta = 1 $, then this reduces to the standard ADMM, where the subproblems are solved exactly (w.p. 1) at each iteration. For the stochastic sADMM, lines 5--9 in Algorithm~1 will read as,
\begin{algorithmic}
	\If {$ \xi\sim \mathbb{U}(0,1) \leq (\delta)^k  $} 
	\State $\tilde{x}_i^{k+1} \leftarrow \arg \min_{x_i}   \mathcal{L}_{i}(x_i,p_{i}^k)$ \Comment Solve NLP
	\Else
	\State $\tilde{x}_i^{k+1} \leftarrow \tilde{x}_i^{k}- \mathcal{M}^{-1}\mathcal{N} (p_{i}^k-p_{i}^{k-1})$\Comment Tangential predictor
	\While {$ \|\nabla_{x_i}\mathcal{L}_{i}(\tilde{x}_{i}^{k+1},p_{i}^{k+1}) \| > D $} \Comment{Corrector steps}
	\State $ \mathcal{M}_{i}  \leftarrow \nabla_{x_i,x_i}^2 \mathcal{L}_{i}(\tilde{x}_i^{k+1},p_i^{k+1}) $
	\State $\Delta x_{i}^c \leftarrow - \mathcal{M}_{i}^{-1} \nabla_{x_i}\mathcal{L}_{i}(\tilde{x}_{i}^{k+1},p_{i}^{k+1})$ 
	\State $ \tilde{x}_{i}^{k+1} \leftarrow \tilde{x}_{i}^{k+1} + \Delta x_{i}^c $
	\EndWhile
	\EndIf
\end{algorithmic}
The convergence properties presented in Section~\ref{sec:convergence} also follow for the stochastic sensitivity-assisted ADMM framework, where $ \epsilon_{i}^{k+1} = 0$, if the subproblems are solved exactly at iteration  $ k+1 $.
By using the tangential predictor, the iterations $ x_{i}^{k+1} $ deviates from the solution manifold (given by \eqref{Eq:SolnBound}). One of the advantages of this approach is that, when we solve the subproblems exactly at some iteration $ k $ randomly, we get back on the solution manifold, and use the tangential predictor using the new solution in the subsequent iterations.  This also enables one to trade-off between  accuracy and computation time. 

\subsection{Linear feature-based architecture}
So far, we have considered a generic nonlinear parametric function $ f(u,x) $. When the parametric functional form $ f(u,x) $ has a linear feature-based architecture, for example,\[  f(u,x) = \phi(u)^{\mathsf{T}}x \] where $ \phi(u) $ is some basis function used to project the features into a higher dimensional space,  then the least square loss function $ \ell(\cdot) $ becomes convex quadratic. One might be interested to note that, in this case the sensitivity update step \eqref{Eq:sensitivity} is exact, i.e. $   \tilde{x}_i^*(p^{k+1})  = x_i^*(p^{k+1})  $  and $ \epsilon_{i}^{k+1} = 0$. 
This implies that the  proposed sensitivity assisted ADMM in this case solves exactly the distributed model fitting problem. The convergence in this case can be obtained by setting $ D=0 $ in Section~\ref{sec:convergence}.

\subsection{Optimal sharing problem}\label{sec:optimal exchange}
The proposed sensitivity-assisted ADMM can also be used on other formulations of ADMM.
For example, consider an optimal exchange problem of the form
\begin{equation}\label{Eq:Allocation}
	\min_{\{ x_{i}\}} \sum_{i=1}^{N}f_{i}(x_{i})  \; \text{s.t.} \; \sum_{i=1}^Nx_{i} = 0
\end{equation}
As shown in \cite[Section 7]{boyd2011ADMM}, such problems  can be solved using the ADMM iterations 
\begin{subequations}
	\begin{align}
		x_{i}^{k+1} &= \arg \min_{x_{i}} \; \mathcal{L}_{i}(x_{i},p_{i}^{k+1}) =  f_{i}(x_{i}) + \lambda^{k\mathsf{T}}x_{i}
		+ \frac{\rho}{2} \|x_{i} - x_{i}^k +\bar{x}^k\|^2, \; \forall i = 1,\dots,N \label{Eq:ResourceAlloc}\\
		\bar{x}^{k+1} &= \frac{1}{N} \sum_{i}^N x_{i}^{k+1}\\
		\lambda^{k+1} &= \lambda^k + \rho \bar x^{k+1}
	\end{align}
\end{subequations}
In this case, the  subproblems \eqref{Eq:ResourceAlloc} are parametric in $ p_{i}^{k+1} $ where now  $ p_{i}^{k+1} := [\lambda^k, \bar{x}^k]^{\mathsf{T}} $. To this end, the proposed sADMM framework and its variants discussed in this section can be used to efficiently solve optimal allocation problems of the form \eqref{Eq:Allocation}. 

Such problem formulations arise commonly in operations research and control problems. In the context of distributed learning, if one were to decompose the learning problem by splitting across the features, as opposed to splitting across the data set, then this results in  an allocation problem formulation, as explained in \cite[Section 8.3]{boyd2011ADMM}.

%% file: experiments.tex
\section{Numerical Experiments}
In this section we demonstrate the use of the proposed sensitivity-assisted ADMM algorithm for distributed model fitting using different examples, and show that by using the sADMM method, we can get similar performance as the traditional ADMM approach, but at significantly less computation time. 
All nonlinear optimization problems in these examples were developed using \texttt{CasADi v3.5.1} \cite{Andersson2019}, which is an open-source tool for nonlinear optimization and algorithmic differentiation. The resulting optimization problems were solved using \texttt{IPOPT v3.12.2} \cite{IPOPTwachter2006IPOPT} with  \texttt{MUMPS} linear solver. All the numerical experiments were performed on a 2.6 GHz processor with 16GB RAM. The source code for the numerical examples presented in this section can be found in the GitHub repository \texttt{https://github.com/dinesh-krishnamoorthy/sADMM}. 

\subsection{Regression example: Multilayer perceptron for combined cycle power plant modelling}

We now illustrate the use of the proposed sensitivity-assisted ADMM to train a multilayer perceptron (MLP) with $ q $ neurons \[ f(u,x) := g_{1}\circ \alpha \circ g_{0} \] where $ g_{0}(u) = w_{0}^{\mathsf{T}}u + b_{0} $, $ g_{1}(\xi) =   w_{1}^{\mathsf{T}}\xi + b_{1} $, $ \xi \in \mathbb{R}^{q} $ is the output of the hidden layer, and $ \alpha(h_{0}) : \mathbb{R} \rightarrow \mathbb{R} $ denotes a nonlinear activation function. 
The unknown parameters $ x $ contains the weights and biases $ w_{0},w_{1}, b_{0} $ and $ b_{1} $, respectively. In this case, the functional form $ f(u,x) $ makes the optimization problem  nonlinear and nonconvex.  In addition, we use $ \ell_1 $ regularization, i.e. $ h(x) = \omega\|x\|_{1} $, making it non-smooth. 

We consider the problem of predicting the net hourly electrical energy output from a combined cycle power plant. The training data set is obtained from the UCI machine learning repository \cite{tufekci2014prediction,tufekci2014combined}. This data set was collected from a real combined-cycle power plant over 6 years, where we have four features, namely, ambient temperature, pressure, relative humidity, and exhaust vacuum with a total of 9568 data points. The data points were normalized such that they have zero mean and standard deviation of 1. 

In this example, we use a network architecture with $ q =5 $ neurons, each with a sigmoid activation function to fit a training data set with 9568 data points consisting of 4 features and 1 label.  The chosen network architecture results in the model parameters $ x \in \mathbb{R}^{31} $, and we use  an $ \ell_{1} $ regularization. Note that the main objective here is to compare the performance of network architecture trained using the ADMM and the sADMM algorithms, and the choice of the hyperparameters itself is not the main focus, which is fixed in all the cases.

The  training data set is decomposed into $ N=4 $ data sets, and the objective is to obtain a global model $ f(u,x) $ using distributed learning with 4 worker nodes. 
We first  solve the MLP learning problem using traditional ADMM, where each subproblem \eqref{Eq:ADMM_xi} is solved as a full NLP at each ADMM iteration. Solution to the subproblem \eqref{Eq:ADMM_x0} was given by the closed-form solution \eqref{Eq:ADMM_x0_explicit_l1}. The ADMM was performed for 200 iterations. Fig.~\ref{Fig:ADMMmetrics} shows the progress of the primal residual,  dual residual, and the augmented Lagrangian by iteration (in blue). Since the subproblems \eqref{Eq:ADMM_xi} were solved exactly, the approximation error of the optimality condition \eqref{eq:sensopt} was zero as shown in the bottom left subplot of Fig.~\ref{Fig:ADMMmetrics}. The  maximum CPU time  for the four worker nodes at each iteration, as well as the overall CPU time for the leaning problem is  also shown in Fig.~\ref{Fig:ADMMmetrics}. 

\begin{figure}
	\centering
	\includegraphics[width=\linewidth]{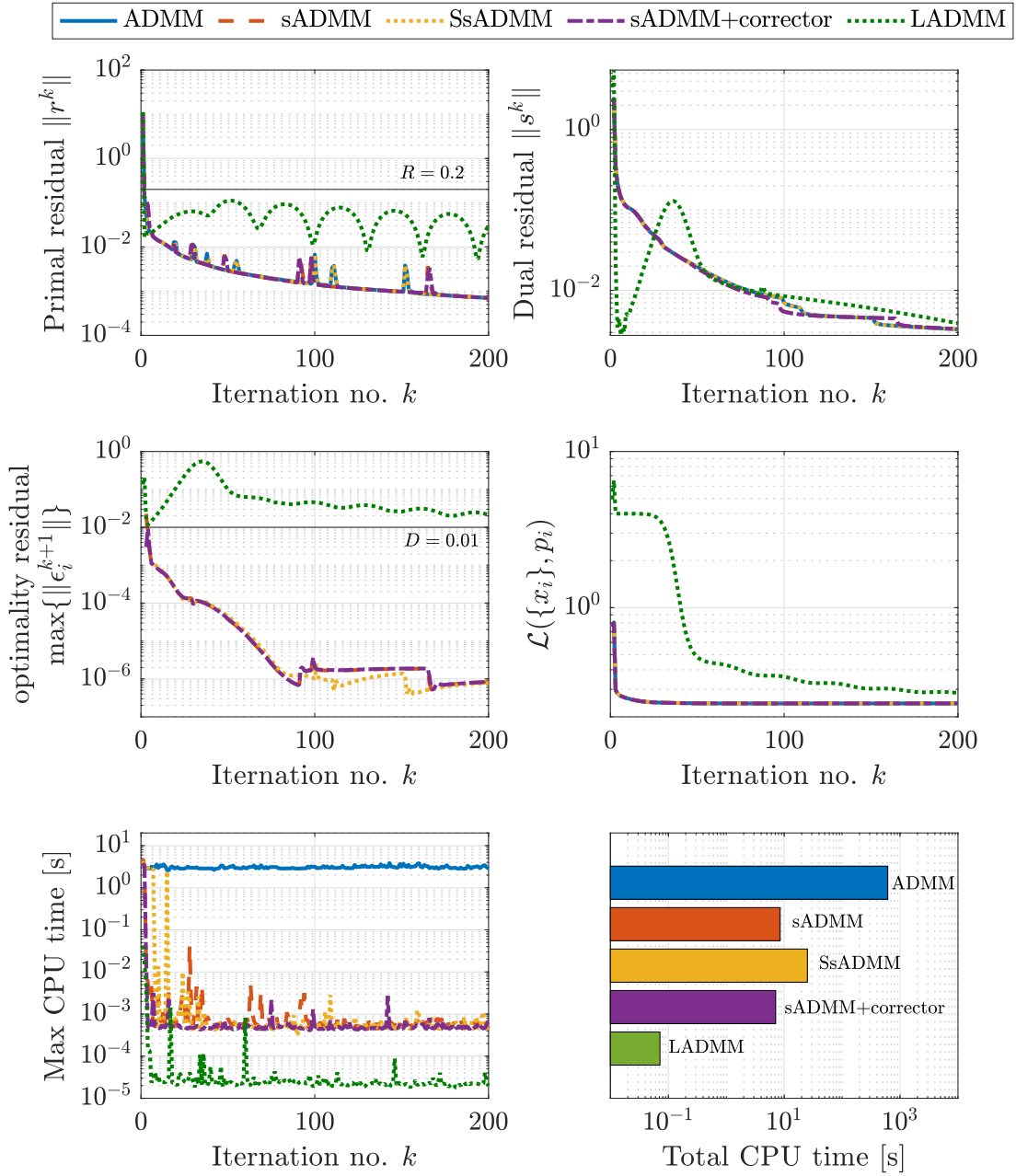}
	\caption{Regression example: Performance metrics of the distributed learning problem using standard ADMM (solid blue line), the proposed sensitivity assisted ADMM (red dashed line), and the stochastic sADMM (yellow dotted line). (a) Primal residual. (b) Dual residual. (c) Approximation error of the optimality condition \eqref{eq:sensopt}. (d) CPU time for the 4 worker nodes.  }\label{Fig:ADMMmetrics}
\end{figure} 
\begin{figure}
	\centering
	\includegraphics[width=\linewidth]{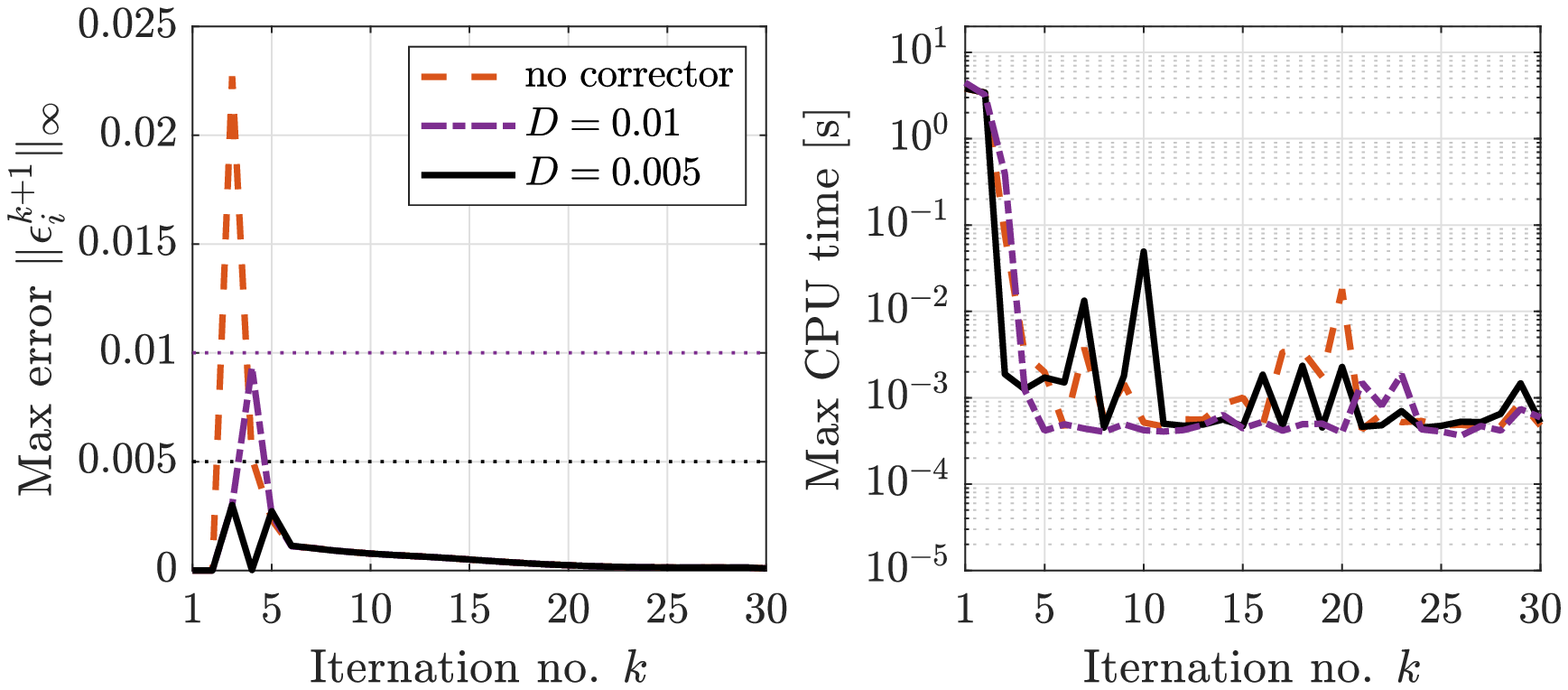}
	\caption{ Comparing the effect of additional corrector steps on the desired optimality condition error bound and the maximum CPU time per iteration.    }\label{Fig:ADMM D}
\end{figure} 


We then solve the same supervised learning problem using the proposed sensitivity-assisted ADMM (sADMM) approach shown in Algorithm~1, where no additional corrector steps were taken. The subproblems \eqref{Eq:ADMM_xi}  were solved exactly if the primal residual was greater than $ R = 0.2$, otherwise it was solved approximately using the tangential predictor \eqref{Eq:sensitivity}.  The progress of the primal residual, dual residual, and the augmented Lagrangian using the proposed sADMM approach, as well as the maximum CPU time for the four worker nodes at each iteration and the total CPU time is compared with the standard ADMM in Fig.~\ref{Fig:ADMMmetrics}, where it can be seen that the proposed sADMM approach (shown in red dashed line) is able to achieve a similar performance (in terms of the augmented Lagrangian, primal residual, and dual residual) but at CPU times approximately five orders of magnitude smaller. 


 With only the predictor steps, we obtained a max error on the optimality residual which is below $ 0.025 $. We argued that  by performing additional corrector steps, we can force  the approximation error to be below a certain user defined bound. If we want this bound to be less than $ D =0.01 $ we can perform corrector steps until $\|\epsilon_{i}^{k+1} \| \le D =0.01$ is satisfied. This is now  shown in the Fig.~\ref{Fig:ADMMmetrics} in purple, where it can be seen that the approximation error bound of $ D=0.01 $ is satisfied, but only at a marginally increased CPU time per iteration.  We also performed the sADMM iterations with $ D = 0.005 $.  The error bound and the CPU time for the sADMM without corrector steps (red dashed lines), sADMM with $ D = 0.01 $ (dot-dashed purple line), and $ D =0.005 $ (black solid) are magnified and shown in Fig.~\ref{Fig:ADMM D}. This clearly shows that by using the proposed sADMM algorithm, one can reduce the overall CPU time of the distributed optimization algorithm, while enforcing a desired  bound on the maximum optimality residual at each iteration.


We also show  the performance of the linearized ADMM (LADMM) with additional regularization term$ 0.5\mu \|x_{i} - x_{i}^k\|^2 $ as described in \cite{qiao2016linearized}, where we set the linearization parameter $ \mu = 10000 $.  This is also shown in Fig.~\ref{Fig:ADMMmetrics} in green  dotted lines. Here  it can be clearly  seen that although the CPU time for LADMM is much lower,  the optimality residual $ \|\epsilon_{i}\| $ and the augmented Lagrangian are worse than the  proposed sensitivity-assisted ADMM.

For the sake of comparison, we also simulate the  stochastic sADMM approach described in Section~\ref{sec:SsADMM}, where the subproblems \eqref{Eq:ADMM_xi} are solved exactly with probability  $ (\delta)^k $ with $ \delta = 0.8 $. The progress of the stochastic sADMM is shown in yellow dotted lines in  Fig.~\ref{Fig:ADMMmetrics}).  The  cumulative CPU time for 200 iterations for the four worker nodes using the different approaches, as well as the total CPU times are summarized in Table~\ref{tb:NN}.

Noting  that in statistical and machine learning problems, the  metric of interest is  also the prediction performance more than the accuracy of the optimization problem, we also compare the statistical fit  of the multilayer perceptron trained using the standard ADMM, linearized ADMM as well as the proposed sADMM  with and without the corrector steps.  
The  mean-squared error and the  $ R^2 $ values  were 0.0598 and 0.9402 for ADMM, and sADMM with and without coorector steps, whereas for LADMM it was 0.0697 and 0.9303, respectively. This shows that the proposed sensitivity assisted ADMM  is able to obtain near identical  performance as the standard ADMM approach, but at significantly lower learning cost. 
\begin{table}[t]
	\begin{center}
		\caption{ Cumulative CPU time [s] for each worker node.}\label{tb:NN}
		\begin{tabular}{c|cccc|c}
		\hline
		& Node 1 & Node 2 & Node 3 & Node 4 & Total \\
		\hline
		ADMM&   548.8437 & 570.3507  &561.3391  &583.0887  &614.06\\
		sADMM & 7.8911  &  6.9010 &   6.9437  &  6.0235    &8.53\\
		SsADMM &24.2328 &  21.7555 &  20.7394 &  21.9679  & 25.18 \\
		sADMM+corr &	6.6735 &   6.2565  &  5.8564  &  5.4422   & 7.14 \\
		LADMM & 0.07&0.0105 & 0.0052 & 0.0043&0.072\\
		\hline
	\end{tabular}
	\end{center}
\end{table}

\subsection{Classification Example: Control policy fitting for Robot navigation}
In this example, we consider the problem of  direct policy approximation, also commonly known as behavioral cloning or imitation learning \cite{bertsekas2019reinforcement,esmaili1995behavioural,osa2018algorithmicperspectiveImitationLearning}, where the the objective is to learn a control policy from  expert demonstrations. We consider a data set obtained from a real wall-following robot navigation with discrete controls, obtained from the UCI machine learning repository \cite{freire2009short}. The mobile robot has 24 ultrasound sensors that together generates four feedback signals, namely, $ y =  \{  \text{front distance, left distance, right distance, back distance}\}$. The robot has four discrete actuation namely, $ u= $\{\enquote{\textit{Move-Forward}},\enquote{\textit{Sharp-Right}},\enquote{\textit{Slight-Right}},\enquote{\textit{Slight-left}}\}. 

The objective is to learn a  control policy $ u = \pi_{x}(y) $ parameterized by $ x $. 
The data set can be assumed to be obtained from four robots operated by four different experts, and the objective is to learn a global policy using the different expert demonstrations  in a  distributed manner. This results in a consensus optimization problem \cite{boyd2011ADMM}. The data set consists of 5456 samples, each with  four real-valued inputs, and one categorical output. For more detailed information about the robot and the data acquisition, the reader is referred to \cite{freire2009short}.

Since the controls are discrete, the control policy is approximated using  a nonlinear multi-class classifier with $ n_{c} = 4 $ classes. For the parametric policy function, we choose a multilayer perceptron with $ n_{u} = 4 $ input neurons, five neurons with sigmoid activation function in the hidden layer, and $ n_{c} =4 $ neurons with soft-max activation function in the output layer. We use a cross entropy loss function, i.e. \[ f_{i}(x_{i}) :=  -\frac{1}{M_{i}}\sum_{j=1}^{M_{i}}\mathbb{I}(u_{j};c)^{\mathsf{T}} \log(\pi_{x_{i}}(y_{j}))\]
where  $ M_{i} $ is the number of data points in subproblem $ i $ and $ \mathbb{I}(y;c) \in \{0,1\}^{n_{c}}$ is an indicator function 
\[ \mathbb{I}(u;c) = \left\{ \begin{matrix}
	1 & \textup{if } u \in c \\
	0 & \textup{otherwise}
\end{matrix}\right.  \]




We use  $ \ell_{1} $ regularization for the policy fitting, i.e., we also consider an additional cost function for subsystem $ i=0 $, which  is given by $ f_{0}(x_{0}): = \|x_{0}\|_{1} $.  In this case, the solution to subproblem $ i=0 $ can be computed explicitly using the  shrinkage operator as shown in \eqref{Eq:ADMM_x0_explicit_l1}.

The training data set is divided in four data chunks, and the approximate policy function is trained in a distributed manner using ADMM with a maximum number of iteration of 600. We  first use the standard ADMM algorithm which serves as our benchmark.  We  then use our proposed sensitivity-assisted ADMM (denoted by sADMM) to train the policy function.  We switch to approximating the subproblems using the tangential predictor if the primal residual is below $ R = 1.2 $. Note that we first use only the tangential predictor without the additional corrector steps. The primal residual, dual residual, augmented Lagrangian of the overall problem, and the maximum approximation error on the optimality condition over the four worker nodes  $ \|\epsilon_{i}^{k+1}\| $ for ADMM (blue solid) and the sADMM (red dashed) are shown in Fig.~\ref{Fig:ADMMmetrics2}, from which it can be seen that using the proposed sADMM method we can achieve a similar performance as the traditional ADMM approach, but at significantly less computation time (about 4-5 order of magnitude lesser).  Note that since the maximum CPU at each iteration is limited by the slowest node,  (cf. Fig.~\ref{Fig:MemoryLocking}), we plot the maximum CPU time among all the four worker nodes at each iteration. 

 With only the predictor steps, we obtained a max error on the optimality residual which is below $ 30 $. We argued that  by performing additional corrector steps, we can force  the approximation error to be below a certain user defined bound. If we want this bound to be less than $ D =2 $ we can perform corrector steps until $\|\epsilon_{i}^{k+1} \| \le D =2$ is satisfied. This is now  shown in the Fig.~\ref{Fig:ADMMmetrics2} in purple, where it can be seen that the approximation error bound of $ D=2 $ is satisfied, but only at a marginally increased CPU time per iteration. 
 
 To clearly show that the user can get the desired error bound $ D $ on the optimality condition a priori, we then performed the sADMM with corrector steps asking for a tighter bound of
 $ D = 1 $.  The error bound and the CPU time for the sADMM without corrector steps (red dashed lines), sADMM with $ D = 2 $ (dot-dashed purple line), and $ D =1 $ (black solid) is  shown in Fig.~\ref{Fig:ADMM D2}.

We also show  the performance of the linearized ADMM (LADMM) with additional regularization term$ 0.5\mu \|x_{i} - x_{i}^k\|^2 $ as described in \cite{qiao2016linearized}, where we set the linearization parameter $ \mu = 1000 $.  This is also shown in Fig.~\ref{Fig:ADMMmetrics2} in green  dotted lines. Here  it can be clearly  seen that although the CPU time for LADMM is much lower,  the optimality residual $ \|\epsilon_{i}\| $ and the augmented Lagrangian are worse than the  proposed sensitivity-assisted ADMM. 



\begin{figure}
	\centering
	\includegraphics[width=\linewidth]{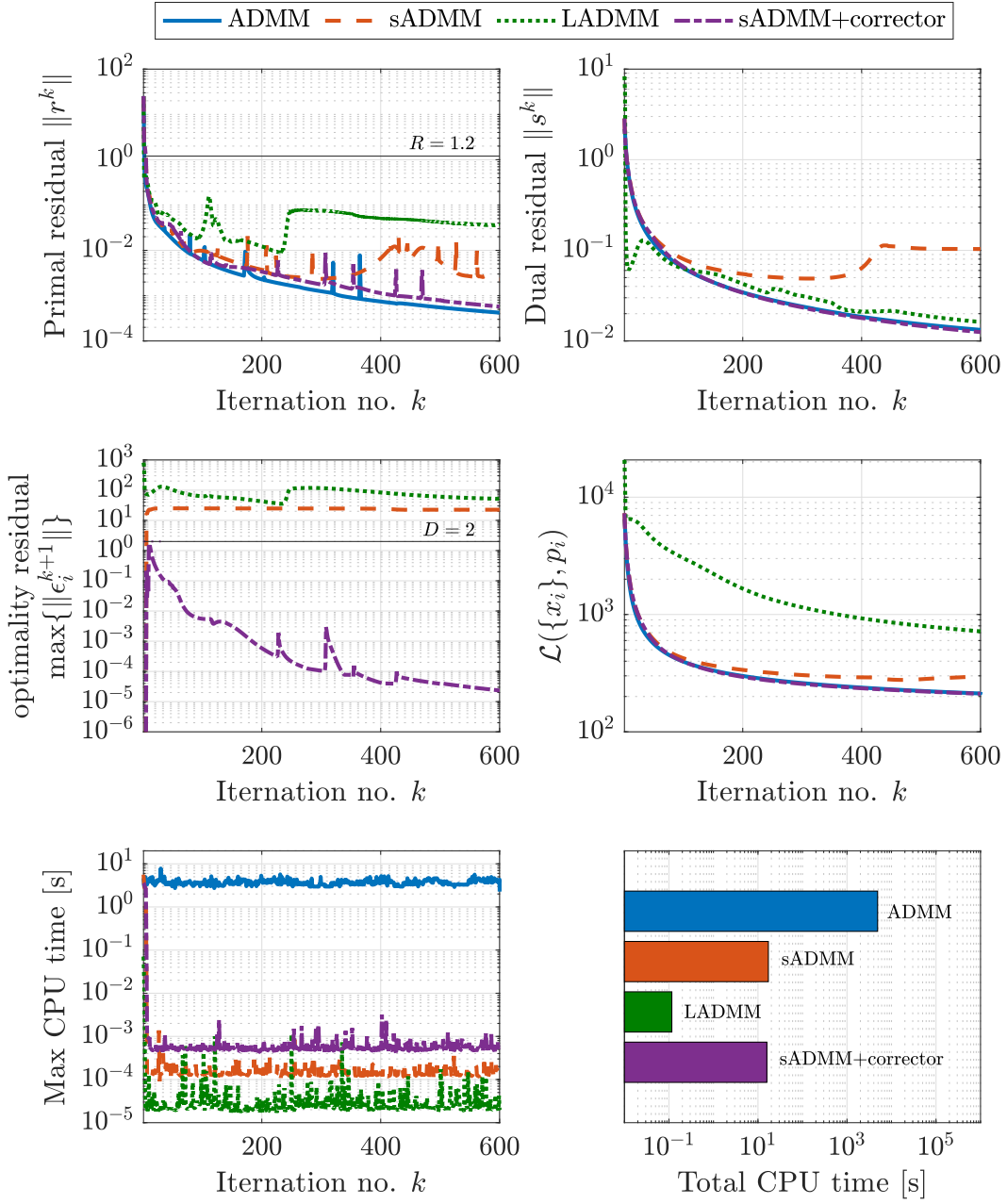}
	\caption{ Performance metrics of  distributed policy fitting using standard ADMM (solid blue line), the proposed sensitivity assisted ADMM (red dashed line),  stochastic sADMM (yellow dotted line), and sADMM with additional corrector steps. }\label{Fig:ADMMmetrics2}
\end{figure}  
\begin{figure}
	\centering
	\includegraphics[width=\linewidth]{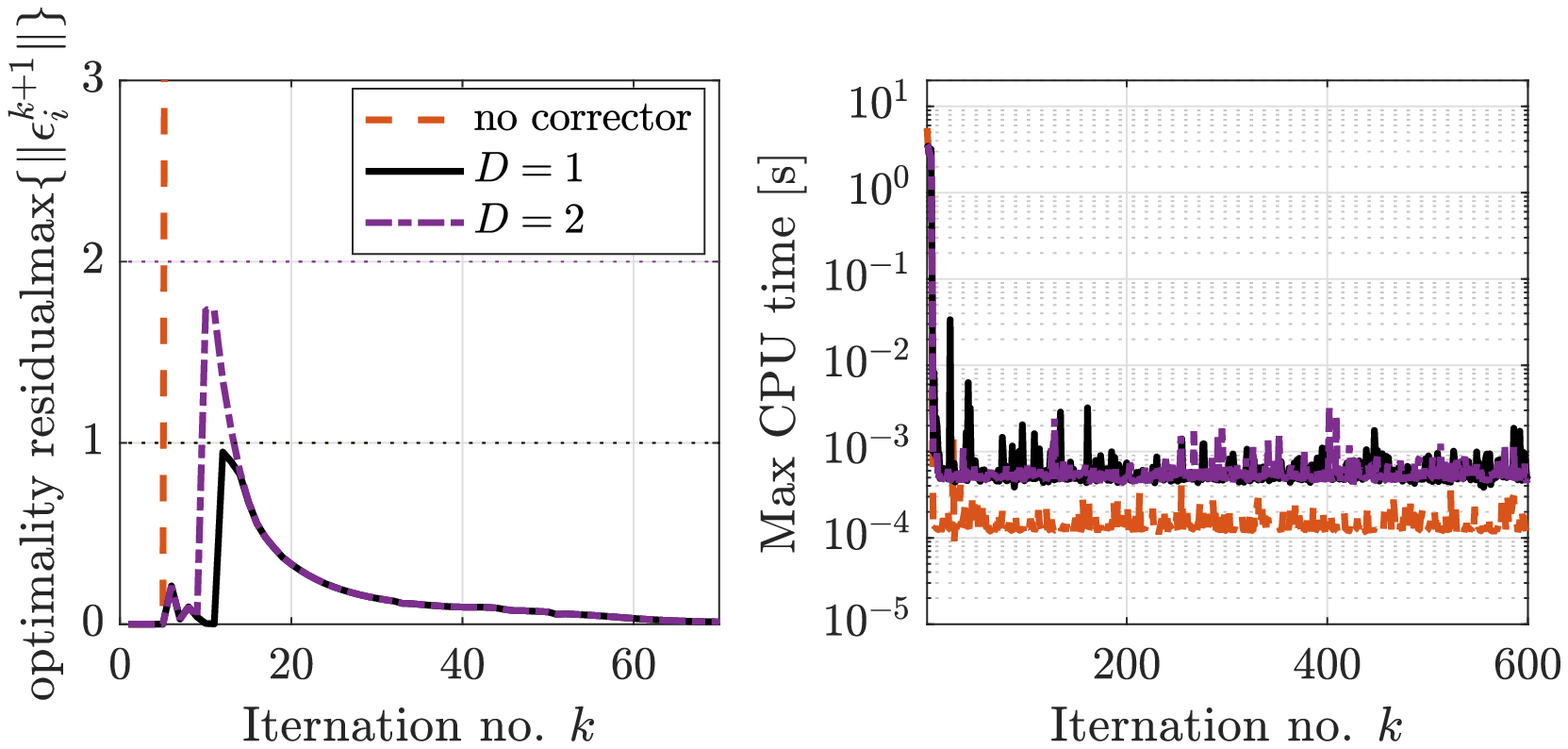}
	\caption{Comparing the effect of additional corrector steps on the desired optimality condition error bound and the maximum CPU time per iteration.  }\label{Fig:ADMM D2}
\end{figure}

\begin{table}[t]
	\begin{center}
		\caption{ Cumulative CPU time [s] for each worker node.}\label{tb:NN2}
		\begin{tabular}{c|cccc|c}
			\hline
			& Node 1 & Node 2 & Node 3 & Node 4 & Total \\
			\hline
			ADMM&    3675.2  &   2505.7  &  3713.3    &2552.1 &   4877.6\\
			sADMM& 16.0677 &  13.1765&   11.8879   &12.4450   &17.1069\\
			LADMM & 14.7198  & 13.5121&   13.3609   &12.9948   &16.1471\\
			sADMM+corr& 0.0684    & 0.0105  &  0.0052  & 0.0043  &  0.0718\\
			\hline
		\end{tabular}
	\end{center}
\end{table}